\documentclass[11pt,letterpaper]{amsart}
\usepackage{amsfonts, amsmath, amssymb, amscd, amsthm, color, graphicx}
\usepackage{hyperref}
\newcommand{\G}{\Gamma (G, \mathcal A)}
\renewcommand{\H}{\mathcal{H}}
\newcommand{\Hl}{\{ H_\lambda\}_{\lambda \in \Lambda}}
\newcommand{\Km}{\{ K_\mu\}_{\mu \in M}}
\newcommand{\Lab}{{\rm Lab}}

\newcommand{\h}{\hookrightarrow_h}

\newcommand{\e}{\varepsilon}
\newcommand{\N}{\mathbb{N}}
\newcommand{\Z}{\mathbb{Z}}

\renewcommand{\d}{{\rm d}}
\newcommand{\dl}{{\widehat{\d}}_\lambda}
\newcommand{\hd}{\widehat{\d}}

\renewcommand{\ll }{\langle\hspace{-.7mm}\langle }
\newcommand{\rr }{\rangle\hspace{-.7mm}\rangle }

\newtheorem{thm}{Theorem}[section]
\newtheorem{cor}[thm]{Corollary}
\newtheorem{lem}[thm]{Lemma}
\newtheorem{prop}[thm]{Proposition}

\theoremstyle{definition}
\newtheorem{defn}[thm]{Definition}
\newtheorem{conv}[thm]{Convention}
\theoremstyle{remark}
\newtheorem{rem}[thm]{Remark}

\newfont{\eufm}{eufm10}

\begin{document}

\title{Acylindrically hyperbolic groups with exotic properties}

\author{Ashot Minasyan}
\address[A. Minasyan]{Mathematical Sciences,
University of Southampton, Highfield, Southampton, SO17 1BJ, United
Kingdom.}
\email{aminasyan@gmail.com}
\author{Denis Osin}
\thanks{The work of the second author has been supported by the NSF grant DMS-1612473.}
\address[D. Osin]{Department of Mathematics, Vanderbilt University,
1326 Stevenson Center, Nashville, TN 37240, USA.}
\email{denis.v.osin@vanderbilt.edu}

\date{}

\begin{abstract}
We prove that every countable family of countable acylindrically hyperbolic groups has a common finitely generated acylindrically hyperbolic quotient. As an application, we obtain an acylindrically hyperbolic group $Q$ with strong fixed point properties: $Q$ has property $FL^p$ for all $p\in [1, +\infty)$, and every action of $Q$ on a finite dimensional contractible topological space has a fixed point. In addition, $Q$ has other properties which are rather unusual for groups exhibiting ``hyperbolic-like'' behaviour. E.g., $Q$ is not uniformly non-amenable and has finite generating sets with arbitrary large balls consisting of torsion elements.
\end{abstract}

\maketitle
\section{Introduction}

An isometric action of a group $G$ on a metric space $S$ is {\it acylindrical} if for every $\e>0$ there exist $R,N>0$ such that for every two points $x,y$ with $\d (x,y)\ge R$, there are at most $N$ elements $g\in G$ satisfying
$$
\d(x,gx)\le \e \quad \mbox{and} \quad \d(y,gy) \le \e.
$$
A group $G$ is called \emph{acylindrically hyperbolic} if it admits a non-elementary acylindrical action on a hyperbolic space; equivalently, $G$ is not virtually cyclic and acts on a hyperbolic space acylindrically with unbounded orbits.

The class of acylindrically hyperbolic groups was introduced in \cite{Osi16} and includes many examples of interest: non-elementary hyperbolic and relatively hyperbolic groups, all but finitely many mapping class groups of punctured closed surfaces, $Out (F_n)$ for $n\ge 2$, finitely presented groups of deficiency at least $2$, most $3$-manifold groups, etc. On the other hand, many aspects of the theory of hyperbolic and relatively hyperbolic groups can be generalized in the context of acylindrical hyperbolicity (see \cite{Osi17} and references therein).

In \cite[Corollary 1.6]{Hull}, Hull proved that any two finitely generated acylindrically hyperbolic groups have a common acylindrically hyperbolic quotient. The main goal of this paper is to prove the following strengthening of his result.

\begin{thm}\label{main}
Every countable family of countable acylindrically hyperbolic groups has a common finitely generated acylindrically hyperbolic quotient.
\end{thm}

We note that neither of the countability assumptions in Theorem \ref{main} can be dropped, see Remark \ref{countrem}.

All corollaries below are obtained by applying Theorem \ref{main} to the family of all non-elementary hyperbolic groups. In particular, all properties mentioned in Corollaries \ref{cor1}--\ref{cor3} can be realized by the same group. Constructing interesting examples as common quotients of countable families of hyperbolic and relatively hyperbolic groups is not a new idea. For instance, it was used in \cite{Osi02} to obtain groups with property (T) whose left regular representation is not uniformly isolated from the trivial representation, and in \cite{Aetal} to obtain groups with strong fixed point properties (see below). Our main contribution is that such a quotient can be made acylindrically hyperbolic.

Recall that a group $G$ is said to have \emph{property $FL^p$} if every affine isometric action of $G$ on an $L^p$-space has a fixed point.
Known examples of groups having property $FL^p$ for all $p\in [1, +\infty)$ include higher rank lattices \cite{BFGM} and
certain Gromov's Monsters \cite{NS}; none of these examples admit non-elementary actions on hyperbolic spaces \cite{GST,Hae}. On the other hand, Yu \cite{Y} proved that every hyperbolic group admits a proper affine action on an $\ell^p$-space for large enough $p$, which is a strong negation of the property $FL^p$. A generalization of Yu's result to relatively hyperbolic groups has been recently obtained by Chatterji and Dahmani in \cite{CD}. It is also known that every acylindrically hyperbolic group $G$ admits an unbounded quasi-cocycle $G\to \ell^p(G)$ for all $p\in [1, +\infty)$ (see \cite{H,HO}), which can be seen as a violation of the ``quasified'' version of $FL^p$.

Motivated by these results, Gruber, Sisto and Tessera \cite{GST} asked whether there exists a group $G$ acting non-elementarily on a hyperbolic space and such that $G$ has property $FL^p$ for all $p\in [1, +\infty)$. The following corollary answers this question affirmatively.

\begin{cor}\label{cor1}
There exists a finitely generated acylindrically hyperbolic group $Q$ such that:
\begin{enumerate}
\item[(a)] $Q$ has property $FL^p$ for all $p\in [1, +\infty)$;
\item[(b)] every action of $Q$ on a contractible Hausdorff topological space of finite covering dimension has a fixed point;
\item[(c)] every simplicial action of $Q$ on a finite dimensional locally finite contractible simplicial complex $\mathcal C$ is trivial, i.e.,
it fixes the whole of $\mathcal C$ pointwise.
\end{enumerate}
\end{cor}

The second and the third claims of the corollary strengthen the main results of \cite{Aetal}, where the  first examples of finitely generated groups that do not admit any fixed point-free actions on finite dimensional contractible Hausdorff topological spaces were constructed. These strong fixed point properties for the acylindrically hyperbolic group $G$ can be contrasted with a theorem of Rips, stating that any hyperbolic group acts properly and cocompactly on a contractible finite dimensional simplicial complex (\cite[Ch. 4, Theorem 1]{G-dlH}).

The same construction produces groups with interesting ``non-uniform'' behaviour. Recall that there are two potentially non-equivalent ways to define uniform non-amenability of a finitely generated group. The first one was suggested  by Shalom in \cite{Sh} and uses the Kazhdan constants of the left regular representations. More precisely, one says that \emph{the left regular representation of a finitely generated group $G$ is uniformly isolated from the trivial representation } if there exists $\e>0$ such that for every finite generating set $X$ of $G$ and every unit vector $v\in \ell^2(G)$, there exists $x\in X$ such that
$$
\| \lambda (x)v-v\|\ge \e,
$$
where $\lambda$ denotes the left regular representation of the group $G$. Another possibility, considered in \cite{A}, is to control the
F\o lner constants: one says that $G$ is \emph{uniformly non-amenable} if there exists $\e>0$ such that for every finite generating set $X$ of $G$, we have
$\mbox{F\o l} (G,X)\ge \e$, where
\begin{equation}\label{Fol}
\mbox{F\o l} (G,X)=\inf \frac{| \partial_X A|}{|A|}.
\end{equation}
Here the infimum is taken over all finite subsets $A\subseteq G$, and
$$
\partial _XA=\{ a\in A \mid ax\notin A\; {\rm for\; some\; } x\in X^{\pm 1}\} .
$$
It is not difficult to show that if the left regular representation of a finitely generated group $G$ is uniformly isolated from the trivial representation, then $G$ is uniformly non-amenable \cite{A}. To the best of our knowledge, it is still an open problem whether the converse is true.

It was proved in \cite{Osi02} that the left regular representations of certain Baumslag-Solitar groups are not uniformly isolated from the trivial representations. These groups were also shown to be not uniformly non-amenable in \cite{A}. On the other hand, it follows from a result of Koubi \cite{Kou} that every non-elementary hyperbolic group is uniformly non-amenable. This is also true for non-elementary relatively hyperbolic groups \cite{Xie}, mapping class groups of closed surfaces of genus $g\ge 1$ and $Out(F_n)$ for $n\ge 2$ \cite{AAS}. Nevertheless, we have the following.

\begin{cor}\label{cor2}
There exists a finitely generated acylindrically hyperbolic group which is not uniformly non-amenable. In particular, its left regular representation is not uniformly isolated from the trivial representation.
\end{cor}

The first step in the proofs of uniform non-amenability of the groups considered in \cite{AAS,Kou,Xie} consists of showing that there exists a constant $C$ such that for any finite generating set $X$ of $G$, one can find an element $g\in G$, of word length $|g|_X\le C$, which is loxodromic with respect to a certain action of $G$ on a hyperbolic space. In particular, such an element $g$ has infinite order.
Our last corollary shows that even this property may fail to be true in a general acylindrically hyperbolic group.

\begin{cor}\label{cor3}
There exists an acylindrically hyperbolic group $Q$ with the following property. For every $r\in \mathbb N$ and every sufficiently large $n\in \mathbb N$, there is a finite generating set $X$ of $Q$ such that all elements $g\in Q$ of length $|g|_X\le r$ have order at most $n$.
\end{cor}

It is worth noting that we do not know whether every finitely generated acylindrically hyperbolic group has uniform exponential growth.
For the definition of uniform exponential growth and a survey of known results we refer to \cite{Mann}.

The paper is organized as follows. In the next section we collect necessary definitions and results about hyperbolically embedded subgroups. Section 3 is devoted to the small cancellation component of our proof. Its main result, Proposition \ref{SCQ}, strengthens the work of  \cite{Hull} and seems to be of independent interest. To apply Proposition \ref{SCQ} in our settings we show that every acylindrically hyperbolic group contains infinite proper hyperbolically embedded subgroups with universal associated hyperbolicity constant. This is the key novelty of our paper, which is discussed in Section 4. The proofs of Theorem \ref{main} and Corollaries~\ref{cor1}--\ref{cor3} are contained in Section 5.

\subsection*{Acknowledgements} The authors thank the anonymous referee for a careful reading of the paper and their comments.

\section{Preliminaries}\label{sec:prelims}

In this section we recall the definition of a hyperbolically embedded collection of subgroups. This notion was introduced in \cite{DGO} and plays a crucial role in our paper.

In what follows, we will use Cayley graphs of groups with respect to generating alphabets which are not necessarily subsets of the group. More precisely, by a \emph{generating alphabet} of a group $G$ we mean an abstract set  $\mathcal A$ together with a possibly non-injective map $\alpha\colon \mathcal A\to G$ such that $G$ is generated by $\alpha(\mathcal A)$ in the usual sense.  Given a word $a_1^{\epsilon_1}\ldots a_k^{\epsilon_k}$,
where $a_1,\dots,a_k \in \mathcal{A}$ and $\epsilon_1,\dots,\epsilon_k \in \{1,-1\}$,
we say that it \emph{represents} an element $g\in G$ if $g=\alpha(a_1)^{\epsilon_1}\cdots \alpha(a_k)^{\epsilon_k}$ in $G$.

By the \emph{Cayley graph} of $G$ with respect to a generating alphabet $\mathcal A$, denoted $\Gamma (G, \mathcal A)$, we mean a graph with the vertex set $G$ and the set of edges defined as follows. For every $a\in \mathcal A$ and every $g\in G$, there is an oriented edge $e_{g,a}=(g, g\alpha(a))$ in $\Gamma (G, \mathcal A)$ labelled by $a$. Note that $\Gamma (G, \mathcal A)$ may have multiple edges if (and only if) $\alpha $ is not injective. By $d_{\mathcal A}$ (respectively, $|\cdot|_{\mathcal A}$) we denote the standard edge-path metric on
$\Gamma (G,\mathcal A)$ (respectively, the word length on $G$ with respect to the generating set $\alpha (A)$).

By abuse of notation, we often identify words in the alphabet $\mathcal A$ and the elements of $G$ represented by them. For two words $U$ and $V$, we write $U \equiv V$ to denote the letter-by-letter equality between them, and $U=_G V$ if these words represent the same element in the group $G$. For a word $W$, $\|W\|$ denotes its length.
If $p$ is a simplicial path in the Cayley graph $\Gamma (G, \mathcal A)$, $p_-$ and $p_+$ will  denote the initial and the terminal vertices of $p$ respectively, and $l(p)$
will denote the length of this path.
The \emph{label} of $p$, denoted by $\Lab(p)$, is the word  obtained by reading off the labels of its oriented edges from $p_-$ to $p_+$.

Suppose that we have a group $G$, a collection of subgroups $\Hl$ of $G$, and a subset $X\subseteq G$ such that $X$ and the union of all $H_\lambda$ together generate $G$. In this case we say that $X$ is a \emph{relative generating set}
of $G$ with respect to $\Hl$. We think of $X$ and subgroups $H_\lambda$ as abstract sets and consider the disjoint unions
\begin{equation}\label{calH}
\mathcal H = \bigsqcup\limits_{\lambda\in \Lambda} H_\lambda\;\;\;  {\rm and}\;\;\; \mathcal A= X \sqcup \mathcal H.
\end{equation}
Obviously $\mathcal A$ is a generating alphabet of $G$, where the map $\alpha \colon \mathcal A\to G$ is induced by the natural inclusions $X\to G$ and $H_\lambda\to G$. Note that $\alpha $ may not be injective.

\begin{conv}
Henceforth we always assume that all generating sets and relative generating sets are symmetric. That is, if $x\in X$, then $x^{-1}\in X$. In particular, every element of $G$ can be represented by a (positive) word in  $\mathcal A$. Given a word $W\equiv a_1\ldots a_k$ in $\mathcal A$, we denote by $W^{-1}$ the word $a_k^{-1}\ldots a_1^{-1}$, where $a_i^{-1}$ are letters of $\mathcal A$: $a_i^{-1}\in X$ whenever $a_i\in X$ and $a_i\in H_\lambda$ whenever $a_i\in H_\lambda$.
\end{conv}

In these settings, we consider the Cayley graphs $\G $ and $\Gamma (H_\lambda, H_\lambda)$, $\lambda\in \Lambda$, and naturally think of the latter as subgraphs of the former. For each $\lambda \in \Lambda $, we introduce an extended metric $\dl \colon H_\lambda \times H_\lambda \to [0, +\infty]$ as follows.

\begin{defn} \label{def:d_hat}
For every $g,h\in H_\lambda $, let $\dl (g,h)$ denote the length of a shortest path in $\G $ that connects the elements $g, h\in H_\lambda $ and contains no edges of $\Gamma (H_\lambda, H_\lambda)$. If no such a path exists, we set $\dl (h,k)=\infty $.
\end{defn}

Clearly $\dl $ satisfies the triangle inequality, where the addition is extended to $[0, +\infty]$ in the natural way.

\begin{defn}\label{hedefn}
A collection of subgroups $\Hl$ of $G$ is \emph{hyperbolically embedded  in $G$ with respect to a subset $X\subseteq G$}, denoted $\Hl \h (G,X)$, if the following conditions hold.
\begin{enumerate}
\item[(a)] The group $G$ is generated by the alphabet $\mathcal A$ defined in (\ref{calH}) and the Cayley graph $\G $ is hyperbolic.
\item[(b)] For every $n\in \mathbb N$ and every $\lambda\in \Lambda$, the ball $\{ h\in H_\lambda\mid \dl(1,h)\le n\}$  contains finitely many elements.
\end{enumerate}
We say  that $\Hl$ is \emph{hyperbolically embedded} in $G$ and write $\Hl\h G$ if $\Hl\h (G,X)$ for some $X\subseteq G$.
\end{defn}

For details and further information on hyperbolically embedded subgroups we refer to \cite{DGO}.
We will need the following result, which is a simplification of \cite[Theorem 1.2]{Osi16}.

\begin{thm}\label{heah}
A group $G$ is acylindrically hyperbolic if and only if it contains a proper infinite hyperbolically embedded subgroup.
\end{thm}

\section{Small cancellation quotients of groups with hyperbolically embedded subgroups}
The proof of Theorem \ref{main} makes use of small cancellation theory over acylindrically hyperbolic groups. The idea of generalizing classical small cancellation to groups acting on hyperbolic spaces is due to Gromov \cite{Gro}.  For hyperbolic groups it was elaborated by Olshanskii in \cite{Ols93}. This approach was generalized to relatively hyperbolic groups by the second author in \cite{Osi10}, and further extended to acylindrically hyperbolic groups by Hull \cite{Hull}.
We begin by recalling necessary definitions.

Let $G$ be a group generated by an alphabet
$\mathcal A$, and let $\mathcal R$ be a symmetrized set of words over $\mathcal A$; that is, we assume that for every $R\in \mathcal R$, $\mathcal R$ contains all cyclic permutations of $R^{\pm 1}$.

Recall that a path $p$ in a metric space $(S,\d)$ is said to be {\it $(\lambda , c)$-quasi-geodesic} for some $\lambda > 0$, $c\ge 0$, if
$$
\d(q_-, q_+)\ge \lambda l(q)-c
$$
for any subpath $q$ of $p$. Further, a word $R$ in $\mathcal A$ is {\it $(\lambda, c)$-quasi-geodesic}, if some (equivalently, any) path labelled by $R$ in $\Gamma (G, \mathcal A)$ is $(\lambda, c)$-quasi-geodesic.

\begin{defn}\label{SCC}
The set $\mathcal R$ satisfies the {\it $C(\e , \mu , \lambda , c, \rho )$ small cancellation condition} (with respect to $\mathcal A$) for some $\e \ge 0$, $\mu >0$, $\lambda >0$, $c\ge 0$, $\rho >0$, if
\begin{enumerate}
\item[(a)] $\| R\| \ge \rho $ for any $R\in \mathcal R$;
\item[(b)] any word $R\in \mathcal R$ is $(\lambda,c)$-quasi-geodesic;
\item[(c)] suppose that for two words
$R, R^\prime \in \mathcal R$ we have
 $R\equiv UV$, $R^\prime \equiv U^\prime V^\prime $, $U^\prime =_G YUZ$
for some words $Y,Z$ in $\mathcal A$ such that
$$
\max \{ \| Y\| ,\| Z\| \} \le \e \quad {\rm and}\quad \max \{ \| U\| , \| U^\prime\| \} \ge \mu \| R\| ;
$$ then $YRY^{-1}=_G R^\prime $.
\end{enumerate}
\end{defn}

Recall that, by our definition, every generating alphabet $\mathcal A$ of a group $G$ comes equipped with a not necessarily injective map $\alpha\colon \mathcal A\to G$. In some cases, if $a, b\in \mathcal A$ represent the same element of $G$ (i.e., if $\alpha(a)=\alpha(b)$), we may want to replace $\mathcal A$ with another generating alphabet $\mathcal A^\prime$ of $G$, where the letters $a$ and $b$ are identified. Clearly passing from $\mathcal A$ to $\mathcal A^\prime$ does not change the word metric on $G$. To prove the main result of this section we will need the following elementary observation.

\begin{lem}\label{AB}
Let $\mathcal A_1$ and $\mathcal A_2$ be generating alphabets of a group $G$ with the corresponding maps $\alpha_i \colon \mathcal A_i\to G$, $i=1,2$. Let $\xi\colon \mathcal A_1\to \mathcal A_2$ be a surjective map such that $ \alpha_1 = \alpha_2 \circ \xi$. Suppose that $\mathcal R_{1}$ is a symmetrized set of words in the alphabet $\mathcal A_1$ satisfying the $C(\e , \mu , \lambda , c, \rho )$ small cancellation condition (with respect to $\mathcal A_1$) for some $\e \ge 0$, $\mu >0$, $\lambda >0$, $c\ge 0$, $\rho >0$. Let $\mathcal R_{2}$ be the set of words in $\mathcal A_2$ obtained from words in $\mathcal R_{1}$ by replacing each letter $a\in \mathcal A_1$ with $\xi(a)\in \mathcal A_2$. Then $\mathcal R_{2}$ satisfies the same $C(\e , \mu , \lambda , c, \rho )$ small cancellation condition (with respect to $\mathcal A_2$).
\end{lem}

\begin{proof}
We extend the map $\xi$ to free monoids $\xi^\ast \colon \mathcal A_1^\ast \to \mathcal A_2^\ast $ in the obvious way. Using surjectivity of $\xi$ it is straightforward to check that $\xi^\ast $ preserves the property of a word to be $(\lambda, c)$-quasi-geodesic as well as property (c) in Definition \ref{SCC} and the claim follows.
\end{proof}

\begin{prop} \label{SCQ} Let $G$ be a  group, $\Hl$, $\Km$ two collections of subgroups of $G$, $X$ subset of $G$. Suppose that $$\Hl\cup \Km \h(G,X).$$ Then there exists $n \in \mathbb N$ and finite subsets $\mathcal F_\lambda\subseteq H_\lambda$, $\lambda \in \Lambda$, such that the following holds.

Let $\mathcal H$, $\mathcal A$ be the alphabets defined by (\ref{calH}) and let $\mathcal W=\{ W_i\}_{i\in I}$ be any set of words in $\mathcal A$ of the form
\begin{equation}\label{W}
W_i\equiv x_i a_{i1} b_{i1}\ldots a_{in}b_{in}
\end{equation}
satisfying the following conditions for all $i \in I$:
\begin{enumerate}
\item[(a)] $x_i\in X$;
\item[(b)] there exist $\alpha=\alpha (i)$ and $\beta=\beta (i)$ in $\Lambda$ such that
$H_{\alpha} \cap {H_\beta} =\{ 1\}$
and $a_{ij}\in H_\alpha\setminus \mathcal F_\alpha$, $b_{ij}\in H_\beta\setminus \mathcal F_\beta$ for all $1\le j\le n$;
\item[(c)] if a letter $c\in \mathcal H$ occurs in $W_i$ for some $i\in I$, then it occurs only in $W_i$ and only once; in addition, $c^{-1}$ does not occur in any word from $\mathcal W$.
\end{enumerate}
Let
\begin{equation}\label{Gbar}
\overline{G}= G/\ll W_i, \, i\in I\rr.
\end{equation}
and
$$
Y=X\cup \left(\bigcup_{\lambda\in \Lambda} H_\lambda\right) \subseteq G.
$$
Then the restriction of the natural homomorphism $\gamma\colon G\to \overline{G}$ to the subset
$Y\cup \left(\bigcup_{\mu\in M} K_\mu\right) \subseteq G$ is injective and $\{ \gamma(K_\mu)\}_{\mu\in M}\h (\overline{G}, \gamma(Y)).$
\end{prop}

\begin{proof}
The proof relies on results from \cite{Hull}. We assume the reader to be familiar with the material discussed in Sections 4 and 5 of \cite{Hull}; to keep this paper reasonably short we do not repeat it here.

We first note that the subsets $\mathcal F_\lambda \subseteq H_\lambda$ can be chosen so that for every $n$ and every $i\in I$, every cyclic permutation of the word $W_i$ satisfies the conditions (W1)--(W4) from \cite[Section 5]{Hull}. Indeed (W1) is obvious,  (W2) can be ensured by taking $\mathcal F_\lambda$ to be the (finite) set of all elements $h\in H_\lambda$ satisfying $\dl (1, h)< 50C$ in the notation of \cite[Section 5]{Hull}, (W3) is guaranteed by (b) and the structure of $W_i$ (the first alternative in the conclusion of (W3) always holds), and (W4) also follows from (b).

From now on, we assume that the subsets $\mathcal F_\lambda \subseteq H_\lambda$ are chosen as explained above. Thus Lemmas 5.1 and 5.2 from \cite[Section 5]{Hull} hold in our settings. Together with our assumption (c), they allow us to repeat the proof of \cite[Proposition 5.3]{Hull} and obtain the following.

\medskip

{\bf Claim. } {\it For every $\e>0$ there exists a constant $M>0$ such that the set $\mathcal R$ of all cyclic shifts of all words $W_i^{\pm 1}$ satisfies the $C(\e, M/n, 1/4, 1, n)$ small cancellation condition.}

\medskip

Although \cite[Proposition 5.3]{Hull} deals with the case $|I|=1$, its proof works almost verbatim in the general case; the only change we need to make is to replace the phrase ``since $\Lab(e)$ only appears once in $W^{\pm 1}$" in the beginning of the third paragraph of the proof with the reference to the condition (c). Note also that \cite[Proposition 5.3]{Hull} proves the $C_1(\e, M/n, 1/4, 1, n)$  condition, which is stronger than $C(\e, M/n, 1/4, 1, n)$ (see \cite[Definition 4.3]{Hull}); however, we do not need this stronger condition in our paper.

Let $$\mathcal K=\bigsqcup_{\mu\in M}K_\mu .$$ Applying Lemma \ref{AB} to the generating alphabets $\mathcal A_1=\mathcal A\sqcup \mathcal K=X\sqcup \mathcal H \sqcup \mathcal K$ and $\mathcal A_2= Y\sqcup \mathcal K $ with the obvious maps $\alpha_1 \colon \mathcal A_1\to G$, $\alpha_2\colon \mathcal A_2\to G$, and $\xi \colon \mathcal A_1\to \mathcal A_2$, we obtain that the set of words $\mathcal R^\prime$ in the alphabet $\mathcal A_2$, obtained from words in $\mathcal R$ by replacing each letter $a\in \mathcal A_1$ with $\xi(a)$, satisfies the same $C(\e, M/n, 1/4, 1, n)$ small cancellation condition.

Note that $\Km\h (G, Y)$  by definition (cf. \cite[Remark 4.26]{DGO}). The set of words $\mathcal R^\prime $ is strongly bounded with respect to the collection of hyperbolically embedded subgroups $\Km$ and the relative generating set $Y$ of $G$ in the terminology of \cite[Section 3, p. 1089]{Hull}. Indeed, in our settings, the latter condition means that only finitely many letters from each $K_\mu $ appear in words from $\mathcal R^\prime$; in our case there are no such letters at all. Let $\e$, $\mu$, $\rho $ be the constants provided by \cite[Lemma 4.4]{Hull} for the group $G$, the hyperbolically embedded collection $\Km$, the relative generating set $Y$,  $\lambda=1/4$, $c=1$ and $N=1$.

By choosing sufficiently large $n$ we can ensure that $M/n\le \mu$ and $n\ge \rho$. Observe that the $C(\e, \mu, \lambda, c, \rho)$ condition becomes stronger as $\mu $ decreases and $\rho $ increases. Thus taking $n$ large enough, we can apply \cite[Lemma 4.4]{Hull} to the set of words $\mathcal R^\prime$  and conclude that the natural homomorphism $G\to \overline{G}$ is injective on the set $Y\cup \left(\bigcup_{\mu\in M} K_\mu\right)$ and we have $\{ \gamma(K_\mu)\}_{\mu\in M}\h (\overline{G}, \gamma(Y)).$
\end{proof}

\begin{rem} Note that in the proof of Proposition \ref{SCQ} we could not apply \cite[Lemma~4.4]{Hull} directly to the set $\mathcal R$,
of all cyclic permutations of words from $\mathcal{W}^{\pm 1}$. Indeed, if  the set $I$ is infinite,
infinitely many letters from $\mathcal{H}$ may appear in the words from $\mathcal{W}$
(because $\alpha(I)$ or $\beta(I)$ may be infinite subsets of $\Lambda$), which would imply that $\mathcal{R}$ is not
strongly bounded with respect to the collection of hyperbolically embedded subgroups
$\{H_\lambda\}_{\lambda\in\Lambda} \cup \Km$.
\end{rem}

\section{Constructing uniformly hyperbolically embedded subgroups} \label{sec:uniform-hyp}

We say that a metric space $(S,\d)$ is \emph{$\delta$-hyperbolic} for some constant $\delta \ge 0$ if it satisfies the \emph{Gromov $4$-point condition}:
for all $x,y,z,w \in S$ one has
\[ \d(x,z)+\d(y,w) \le \max\{\d(x,y)+\d(z,w),\d(x,w)+\d(z,y)\}+\delta.\] This condition is well-known to be equivalent to other definitions of hyperbolicity
(see \cite[Ch. 2, {\S\S} 2.4, 2.3, 2.21] {G-dlH}).
We say that a metric space is \emph{hyperbolic} if it is $\delta$-hyperbolic for some $\delta \ge 0$.

\begin{defn} Let $\Sigma$ be a graph and let $C$ be a natural number. The \emph{$C$-expansion} of $\Sigma$ is a graph $\Xi$ obtained from
$\Sigma $ by adding an edge between every two vertices $u,v$ that are at most $C$ apart in $\Sigma$.
\end{defn}

\begin{lem}\label{lem:exp-hyp}
Let $\Sigma $ be a connected $\delta$-hyperbolic graph, let $C\in \mathbb N$ and let $\Xi$ be the $C$-expansion of $\Sigma$. Then $\Xi $ is
$(\delta/C +6)$-hyperbolic.
\end{lem}

\begin{proof} Let $\d_\Sigma$ and $\d_\Xi$ denote the standard path metrics on $\Sigma$ and $\Xi$ respectively. Let $u,v$ be any two vertices of $\Xi$.
The definition of $\Xi$ implies that $u,v$ are also vertices of $\Sigma$ and
\begin{equation}\label{eq:dist_ineq}
\d_{\Sigma}(u,v) \le C \d_\Xi(u,v) \mbox{ and } \d_{\Xi}(u,v)\le 
\frac{1}{C} \d_{\Sigma}(u,v)+1.
\end{equation}

Now consider arbitrary points $x,y,z,w$ in $\Xi$, and choose vertices $x',y',z',w'$ of $\Xi$ such that $\d_{\Xi}(x,x') \le 1/2, \ldots , \d_{\Xi}(w,w') \le 1/2$. Then, in view of \eqref{eq:dist_ineq}, we have

\[
\d_\Xi(x,z)+\d_\Xi(y,w)\le \d_\Xi(x',z')+1+\d_\Xi(y',w')+1\le \frac1C \left(\d_{\Sigma}(x',z')+ \d_\Sigma(y',w') \right)+4.
\]
Recalling that $\Sigma$ is $\delta$-hyperbolic and using \eqref{eq:dist_ineq} again, we get
\[
\begin{array}{rcl}
\d_\Xi(x,z)+\d_\Xi(y,w)& \le & \frac1C \left(\max\{\d_{\Sigma}(x',y')+ \d_\Sigma(z',w'),\d_{\Sigma}(x',w')+ \d_\Sigma(z',y')\} +\delta \right)+4
\\&&\\& \le & \max\{\d_{\Xi}(x',y')+ \d_\Xi(z',w'),\d_{\Xi}(x',w')+ \d_\Xi(z',y')\} +\frac{\delta}{C}+4
\\&&\\& \le & \max\{\d_{\Xi}(x,y)+ \d_\Xi(z,w),\d_{\Xi}(x,w)+ \d_\Xi(z,y)\} +\frac{\delta}{C}+6.
\end{array}
\]
Hence the graph $\Xi$ satisfies Gromov $4$-point condition with the constant $(\delta/C +6)$.
\end{proof}

Recall that a subset $Q$ of a geodesic metric space $S$ is $\e$-\emph{quasiconvex}, for some $\e \ge 0$,
if for any two points $x,y \in Q$ any geodesic path joining $x$ and $y$ in
$S$ lies within the closed $\e$-neighborhood of $Q$ in $S$.

\begin{lem}\label{lem:qc_in_exp} There exists a constant $\varkappa \ge 0$ such that for any $\e, \delta \ge 0$ the following holds.
Let $\Sigma$ be a connected $\delta$-hyperbolic graph and let $Q$ be an $\e$-quasiconvex subset of vertices of $\Sigma$, for some $\delta,\e \ge 0$. Then for each positive integer $C\ge \max\{\delta, \e\}$, the $C$-expansion $\Xi$ of $\Sigma$ is
$7$-hyperbolic and $Q$ is $\varkappa$-quasiconvex in $\Xi$.
\end{lem}

\begin{proof} Suppose that $C \in \N$, $C \ge \max\{\delta, \e\}$.
Then $\Xi$, the $C$-expansion of $\Sigma$ equipped with the standard path metric $\d_\Xi$, is $7$-hyperbolic by Lemma \ref{lem:exp-hyp}.
Let $\Sigma_C$ denote the geometric realization of the graph $\Sigma$ where all the edges are assumed to be isometric to the interval $[0,1/C]$. Thus the metric $\d_{\Sigma_C}$  on $\Sigma_C$, can be obtained by rescaling the standard path metric $\d_\Sigma$ on $\Sigma$: for all $x,y \in \Sigma$ one has $\d_{\Sigma_C}(x,y)=\frac1C \d_\Sigma(x,y)$.
Obviously $\Sigma_C$ is $1$-hyperbolic (as $\delta/C \le 1$) and $Q$ is $1$-quasiconvex in $\Sigma_C$ (as $\e/C \le 1$).

The inequalities \eqref{eq:dist_ineq} clearly imply that for all vertices $u,v$ of $\Sigma$ we have
\begin{equation}\label{eq:Xi-Sigma_C}
   \d_{\Sigma_C}(u,v) \le \d_\Xi(u,v) \le \d_{\Sigma_C}(u,v)+1,
\end{equation}
i.e., the natural identification of the vertex set of $\Xi$ with the vertex set of $\Sigma_C$ is a $(1,1)$-quasi-isometry.

Thus for every geodesic path $p$ connecting two vertices in $\Xi$, its vertex set forms the image of a (discrete) $(1,1)$-quasi-segment in $\Sigma_C$, using the terminology of
\cite[Ch. 5, \S~1.2]{G-dlH}. Therefore, by \cite[Ch. 5, \S~1.6]{G-dlH}, this vertex set lies within an $H$-neighborhood of any geodesic path in $\Sigma_C$
which has the same endpoints as $p$, where $H=H(1,1,1)$ is some global constant.

Now, let $p$ be a geodesic path connecting two vertices $a, b \in Q$ in $\Xi$, and let $x$ be any point of $p$. Then there is a vertex $u$ of $p$ such that $\d_\Xi(x,u) \le 1/2$.
Let  $q$ be a geodesic between $a$ and $b$ in $\Sigma_C$. By the above discussion, $u$ lies in the $H$-neighborhood of a vertex $w \in q$ in $\Sigma_C$. On the other hand, the $1$-quasiconvexity of $Q$ in $\Sigma_C$ implies that there is a vertex $v \in Q$ such that $\d_{\Sigma_C}(w,v)\le 1$, hence $\d_{\Sigma_C}(u,v)\le H+1$.
Recalling \eqref{eq:Xi-Sigma_C}, we obtain
$\d_{\Xi}(u,v) \le H+2$, hence $\d_\Xi(x,v) \le H+5/2$.

Thus we can conclude that $Q$ is $\varkappa$-quasiconvex in $\Xi$,
where $\varkappa=H+5/2$ is a global constant (independent of $\delta$, $\e$, etc.), and the lemma is proved.
\end{proof}

The following observation will be useful.

\begin{lem}\label{lem:add_gens-hyp_emb} Suppose that $G$ is a group, $X_1 \subseteq G$ and $\Hl$ is a collection of subgroups in $G$. Let $X_2 \subseteq G$ be a subset
such that $X_1 \subseteq X_2 \subseteq \langle X_1 \rangle$ and $\sup\{|x|_{X_1}\mid x \in X_2\}<\infty$. Then $\Hl \h (G,X_1)$ if and only if $\Hl \h (G,X_2)$.
\end{lem}

\begin{proof} Let $\H=\bigsqcup_{\lambda \in \Lambda} H_\lambda$. Obviously the set $X_1 \sqcup \H$ generates $G$ if and only if the set $X_2 \sqcup \H$ generates $G$,
as $\langle X_1 \rangle=\langle X_2 \rangle$.
Moreover, the assumptions $X_1 \subseteq X_2$ and $\sup\{|x|_{X_1}\mid x \in X_2\}=C<\infty$ imply that the natural inclusion of the Cayley graph
$\Gamma(G,X_1\sqcup \H)$ in $\Gamma(G,X_2\sqcup \H)$ is $C$-bi-Lipschitz. Hence  the Cayley graph
$\Gamma(G,X_1\sqcup \H)$ is hyperbolic if and only if $\Gamma(G,X_2\sqcup \H)$ is hyperbolic (cf. \cite[Ch. 5, \S{} 2.12]{G-dlH}).

For each $\lambda \in \Lambda$ and $i\in \{1,2\}$ let $\hd_\lambda^i:H_\lambda\times H_\lambda \to [0,+\infty]$ denote the metric,
given by  Definition \ref{def:d_hat},
coming from the Cayley graph $\Gamma(G,X_i\sqcup \H)$. Note that $\hd_\lambda^2(1,h)\le \hd_\lambda^1(1,h) \le C\hd_\lambda^2(1,h)$ for all $h \in H_\lambda$.
Therefore the metric $\hd_\lambda^1$ is locally finite on $H_\lambda$ if and only if $\hd_\lambda^2$ is locally finite on $H_\lambda$.
Thus  $\Hl \h (G,X_1)$ if and only if $\Hl \h (G,X_2)$.
\end{proof}

We will also make use of \cite[Corollary 2.4]{KR} (see also \cite[Proposition 2.3]{KR}), which we state in a much simplified form below (in the notation of \cite{KR}, we have $L=M_1=1$ and $M_2=M$).

\begin{lem}\label{KR}
For every $\delta_0, M\ge 0$, there exists $\delta_1\ge 0$ such that the following holds. Let $\Theta$ be a graph obtained from a connected $\delta_0$-hyperbolic graph $\Xi $ by adding edges. Suppose that for any two vertices $x,y$ of $\Xi $, connected by an edge in $\Theta$, and any geodesic $p$ in $\Xi$ going from $x$ to $y$, the diameter of $p$ in $\Theta$ is at most $M$. Then $\Theta$ is $\delta_1$-hyperbolic.
\end{lem}

Suppose that a group $G$ acts on a hyperbolic metric space $(S,\d)$.
This action is said to be \emph{non-elementary} if for some $s \in S$ the orbit $Gs$ has more than two limit points on the
boundary $\partial S$ (see \cite[\S~8.2.D]{Gro}).
An element $g \in G$ is \emph{loxodromic} if for any point $s \in S$ the map $\Z \to S$, defined by
$n \mapsto g^n s$, is a quasi-isometry. Clearly the order of any loxodromic element is infinite.

Recall that every acylindrically hyperbolic group contains a maximal finite normal subgroup called the \emph{finite radical} of $G$ and denoted $K(G)$ (see \cite[Theorem 2.24]{DGO}).

\begin{lem}\label{ab}
There exists  a constant $D$ with the following property. Let $G$ be an acylindrically hyperbolic group with trivial finite radical. Then there exist elements $a,b \in G$ of infinite order and a generating set $X$ of $G$ satisfying the following conditions:
\begin{enumerate}
\item[(a)] $\{ \langle a\rangle, \langle b\rangle \}\h (G,X)$;
\item[(b)] The Cayley graph $\Gamma (G, X\sqcup \langle a\rangle\sqcup \langle b\rangle)$ is $D$-hyperbolic.
\end{enumerate}
\end{lem}

\begin{proof} By \cite[Theorem 1.2]{Osi16} there exists a generating set $Y$ of $G$ such that the Cayley graph $\Gamma(G,Y)$ is
$\delta$-hyperbolic, for some $\delta \ge 0$, and the natural
action of $G$ on it is acylindrical and non-elementary.
Since the finite radical of $G$ is assumed to be trivial, $G$ is a `suitable' subgroup of itself, in the terminology of \cite{Hull}.
Therefore we can use \cite[Corollary 5.7]{Hull} to find loxodromic
elements $a,b \in G$ such that $\{\langle a\rangle,\langle b \rangle\} \h (G,Y)$.

Now, by \cite[Lemma 2.5]{AMS} the cyclic subgroups $\langle a \rangle$, $\langle b \rangle$ are $\e$-quasiconvex in $\Gamma(G,Y)$, for some $\e \ge 0$.
Choose any $C \in \N$ with $C \ge \max\{\delta,\e\}$, and let
$$
X=\{ x \in G \mid |x|_Y \le C\}.
$$

Observe that the Cayley graph $\Gamma(G,X)$ is the $C$-expansion of the Cayley graph $\Gamma(G,Y)$, so, according to Lemma \ref{lem:qc_in_exp},
$\Gamma(G,X)$ is $7$-hyperbolic  and the subgroups $\langle a \rangle$, $\langle b \rangle$ are $\varkappa$-quasiconvex in it
(where $\varkappa$ is the global constant from that lemma).
Moreover, Lemma~\ref{lem:add_gens-hyp_emb} ensures that $\{\langle a \rangle,\langle b \rangle\} \h (G,X)$, so it remains to show that
$\Gamma (G, X\sqcup \langle a\rangle\sqcup \langle b\rangle)$ is $D$-hyperbolic for some global constant $D \ge 0$.

Let $\delta_0=7$ and $M=2\varkappa+1$. We will now check that the assumptions of Lemma~\ref{KR} are satisfied, where $\Xi=\Gamma(G,X)$ and
$\Theta=\Gamma(G,X \sqcup \langle a \rangle\sqcup \langle b \rangle)$.
Let $u,v \in G$ be two adjacent vertices of $\Theta$, let $p$ be any geodesic joining $u$ with $v$ in $\Xi$, and let $w,z$ be arbitrary points on $p$.

If $p$ is a single edge then $\d_\Theta(w,z)\le \d_\Xi(w,z)\le 1 \le M$.
Otherwise, $p$ must be labelled by a power of $a$ or $b$. Let us consider the case when $p$ is labelled by
$a^n$, for some $n \in \Z$, as the other case is similar. Then the $\varkappa$-quasiconvexity of $\langle a \rangle$ in $\Xi$ implies that there exist $k,l \in \Z$ such that
$\d_\Xi(w,u a^k) \le \varkappa$ and $\d_\Xi(z,u a^l)\le \varkappa$. Moreover, since $(ua^k)^{-1}(ua^l)=a^{l-k} \in \langle a \rangle$, the vertices
$ua^k$ and $ua^l$ are adjacent in $\Theta$. Therefore
\[
\begin{array}{rcl}
\d_\Theta(w,z) &\le &\d_\Theta(w,ua^k)+\d_\Theta(ua^k,ua^l)+\d_\Theta(ua^l,z) \\&&\\ & \le &\d_\Xi(w,ua^k)+1+\d_\Xi(ua^l,z)\le 2\varkappa+1=M.
\end{array}
\]
Thus the universal constant $M$ is an upper bound for the diameter of $p$ in $\Theta$. Now we can apply Lemma \ref{KR} to find a universal constant $D \ge 0$ such that
the Cayley graph $\Gamma(G,X \sqcup \langle a \rangle\sqcup \langle b \rangle)=\Theta$ is $D$-hyperbolic, as required.
\end{proof}

\begin{lem}\label{Gi}
Let $\{ G_i\}_{i\in I}$ be a collection of groups. Suppose that for every $i\in I$ we have a collection of subgroups $\{H_{ij}\}_{j\in J_i}$ and
a generating  set $X_i$  of $G_i$, relative to $\{H_{ij}\}_{j\in J_i}$, such that
\begin{enumerate}
\item[(a)] $\{H_{ij}\}_{j\in J_i}\h (G_i, X_i)$, for all $i \in I$, and
\item[(b)] the hyperbolicity constants of the Cayley graphs $\Gamma (G_i, X_i\sqcup \mathcal H_i)$, where $\mathcal H_i=\bigsqcup_{j\in J_i} H_{ij}$, $i \in I$,
 are uniformly bounded.
\end{enumerate}
Then the collection $\{ H_{ij}\mid  i\in I, j\in J_i \}$ is hyperbolically embedded in the free product $G=\ast_{i\in I} G_i$ with respect to $X=\bigcup _{i\in I} X_i$.
\end{lem}

\begin{proof} Suppose that the Cayley graphs $\Gamma(G_i,X_i\sqcup \H_i)$ are $\delta$-hyperbolic, for some fixed $\delta \ge 0$, and all $i \in I$.
After increasing $\delta$, we can assume that $\delta \ge 1$ and geodesic triangles in each of these Cayley graphs are $\delta$-slim, i.e, each side of such a triangle
is contained in the $\delta$-neighborhood of the two other sides (see \cite[Ch. 2, \S{} 3.21]{G-dlH}).
Let us first show that all geodesic triangles in $\Gamma(G,X \sqcup \H)$ are also $\delta$-slim.

Let us first establish some convenient terminology. Given a path $p$ in $\Gamma(G,X \sqcup \H)$, its label is a concatenation of
subwords, each of which is written over the alphabet $X_i \sqcup \H_i$, for some $i \in I$.
Given any  $i \in I$, a \emph{$G_i$-component of $p$} is a subpath of $p$ labelled by a non-empty word over the alphabet
$X_i \sqcup \H_i$, which is not contained in a larger subpath of $p$ with this property.

Let us now make the following observation, which is an immediate consequence of the uniqueness of normal forms of elements in free products.

\begin{rem}\label{rem:observ} If $p$ and $q$ are two simple paths in $\Gamma(G,X \sqcup \H)$ with $p_-=q_-$ and $p_+=q_+$ then $p=p_1p_2\dots p_n$,
$q=q_1q_2\dots q_n$, where for each $k \in \{1,\dots,n\}$ there is $i_k \in I$ such that $p_k$ and $q_k$ are $G_{i_k}$-components of $p$ and $q$ respectively, and
$(p_k)_-=(q_k)_-$, $(p_k)_+=(q_k)_+$.
\end{rem}

Suppose that  $\Delta$ is a simplicial geodesic triangle in $\Gamma(G,X \sqcup \H)$, and $p$ is any side of it. Let $q$ and $r$ denote the two remaining sides of
$\Delta$. Without loss of generality we can re-orient $p,q,r$ so that $p_-=q_-$, $q_+=r_-$ and $r_+=p_+$.
Let $x$ be any point of $p$. Then $x$ belongs to a $G_i$-component $p_i$ of $p$, for some $i \in I$. Let $v$ be the first vertex of $q$ which also belongs to $r$
(see Figure \ref{fig:triangle}), and let
$q'$, $r'$ be the subpaths of $q$ and $r$ respectively, such that $q'_-=q_-$, $q'_+=v=r'_-$ and $r'_+=r_+$. Note that the paths $p$,$q$ and $r$ are simple because
they are geodesic in $\Gamma(G,X \sqcup \H)$, and the choice of $v$ implies that the path $s=q'r'$ is also simple. Since $p_-=s_-$ and $p_+=s_+$, we can use
Remark~\ref{rem:observ} to find a $G_i$-component $s_i$, of $s$, such that $(s_i)_-=(p_i)_-$ and $(s_i)_+=(p_i)_+$.

\begin{figure}[!ht]
  \begin{center}
   \includegraphics{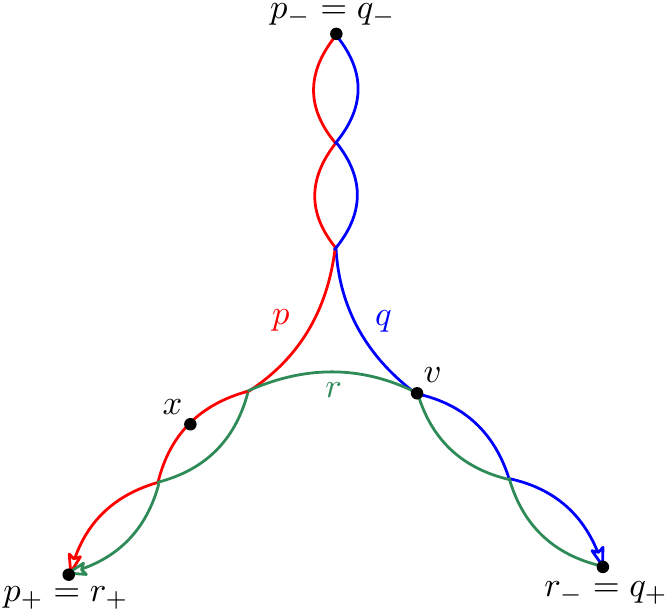}
  \end{center}
  \caption{A generic geodesic triangle $\Delta$ in $\Gamma(G,X \sqcup \H)$}\label{fig:triangle}
\end{figure}

Now, by construction, $s_i$ is either a subpath of $q$ or $r$ or it is a concatenation of a subpath of $q$ ending at $v$ with a subpath of $r$ starting at $v$.
Thus $p_i$ is geodesic and $s_i$ is the union of at most two geodesics in the copy of $\Gamma(G_i, X_i\sqcup \H_i)$ in $\Gamma(G,X \sqcup \H)$ based at $v$.
The $\delta$-slimness of geodesic triangles in $\Gamma(G_i, X_i\sqcup \H_i)$ implies that $x$ is at most $\delta$ away from a point of $s_i$ in $\Gamma(G,X \sqcup \H)$. But $s_i$ is contained in the union of $q$ and $r$, so the geodesic triangle $\Delta$ is $\delta$-slim. Since the latter is true for an arbitrary simplicial geodesic triangle in
$\Gamma(G,X \sqcup \H)$, we can conclude that this Cayley graph is hyperbolic by \cite[Ch. 2, \S{} 3.21]{G-dlH}.

Now, choose any $i \in I$ and $j \in J_i$ and consider the subgroup $H=H_{ij}$, of $G_i \leqslant G$. Suppose that $n \in \N$, $h \in H\setminus \{1\}$,
$p$ is the edge joining $1$ and $h$ and labelled by $h$ in $\Gamma(G,X \sqcup \H)$, and $q$ is any path from $1$ to $h$ in this graph, of length at most $n$,
which avoids the edges from the copy of $\Gamma(H, H)$.
Let $q'$ be a simple path with the same endpoints as $q$, whose set of edges is a subset of the set of edges of $q$ (we can get $q'$ from $q$ by removing all maximal subpaths
of $q$ which start and end at the same vertex). Then, according to Remark~\ref{rem:observ}, $q'$ must consist of a single $G_i$-component (since this is true for $p$ and
$q'_-=1=p_-$, $q'_+=h=p_+$). Thus the
path $q'$ lies inside the Cayley graph $\Gamma(G_i, X_i\sqcup \H_i)$. Clearly the length of $q'$ does not exceed $n$ and it still avoids the edges of $\Gamma(H, H)$.
Recalling that $\{H_{ik}\}_{k\in J_i}\h (G_i, X_i)$, we see that there are only finitely many possibilities for the element $h$ (for a given $n$).
Since this is true for any $n \in \N$ and arbitrary $H_{ij}$, we can conclude that $\{ H_{ij}\mid i\in I, j\in J_i\}\h (G,X)$, as claimed.
\end{proof}

\section{Proofs of the main results}\label{sec:proofs}

We begin with the proof of Theorem \ref{main}. Recall that for an acylindrically hyperbolic group $G$, $K(G)$ denotes its finite radical, i.e., $K(G)$ is the unique maximal finite normal subgroup of $G$.
It is shown in \cite[Lemma 3.9]{MO} (see also \cite[Lemma~1]{MO-erratum}) that the quotient group $G/K(G)$ is also acylindrically hyperbolic and has trivial finite
radical.

\begin{proof}[Proof of Theorem \ref{main}]
Let $G_k$, $k\in \mathbb N$, be a countable set of countable acylindrically hyperbolic groups. Passing to $G_k/K(G_k)$ if necessary, we can assume that the finite radicals of all groups $G_k$ are trivial.

By Lemma \ref{ab}, there exist infinite cyclic subgroups $A_k,B_k\le G_k$ and generating sets $X_k\subseteq G_k$ such that $\{ A_k,B_k\} \h (G_k, X_k)$, and the hyperbolicity constant of the Cayley graph $\Gamma (G_k, X_k\sqcup A_k\sqcup B_k)$ is bounded from above by a universal constant for all $k\in \mathbb N$.

Let $X=\bigcup_{k\in \mathbb N} X_k$ and $G=G_1\ast G_2\ast \cdots $. Note that $X$ generates $G$ and, by Lemma \ref{Gi}, we have $\{ A_k, B_k\}_{k\in \mathbb N} \h (G,X)$.
Combining \cite[Corollary 5.3]{AMS} with \cite[Corollary 3.12]{AMS} we can find an infinite virtually cyclic subgroup $K \leqslant G$ such that
$\{K\} \cup \{ A_k, B_k\}_{k\in \mathbb N} \h (G,X)$.
Let $X=\{ x_1, x_2, \ldots\}$ and let $I=\mathbb N\times \mathbb N$. If $X=\{ x_1, \ldots, x_m\}$, we define $x_{m+1}=x_{m+2}=\dots= 1$ to uniformize our notation. For every $i=(s,t)\in I$,  we consider a word of the form
$$
W_i \equiv x_{s}a_{i1}b_{i1}\ldots a_{in}b_{in},
$$
where $a_{ij}\in A_{t}$, $b_{ij}\in B_{t}$  for all $1\le j\le n$. Note that the intersection of distinct subgroups from a hyperbolically embedded collection is always finite by \cite[Proposition 4.33]{DGO} and therefore $A_k\cap B_k=\{ 1\}$ for all $k$. By Proposition \ref{SCQ}, applied to the group $G$, the set $X$, and the collections
$\Hl=\{ A_k\}_{k \in \N} \cup \{B_k\}_{k\in \mathbb N}$ and $\Km=\{ K\}$, we can chose $n\in \mathbb N$ and finite subsets
$\mathcal{F}_\lambda \subset H_\lambda$, $\lambda \in \Lambda$,
so that the conclusion of the proposition holds as long as the elements $a_{ij}, b_{ij}$ are chosen to satisfy conditions (b) and (c) from the proposition. Since
$A_k$ and $B_k$ are infinite cyclic, $k \in \N$, such a choice is clearly possible.

Let $\overline G$ be the quotient group defined by (\ref{Gbar}). By Proposition \ref{SCQ}, the image of $K$ in $\overline{G}$ is an infinite proper hyperbolically embedded subgroup. Therefore $\overline{G}$ is acylindrically hyperbolic by Theorem \ref{heah}. On the other hand, for every $k\in \mathbb N$, the relations $W_i=1$,
$i\in \mathbb N \times \{k\}$ guarantee that the restriction of the natural homomorphism $G\to\overline{G}$ to $G_k$ is surjective, i.e., $\overline{G}$ is a quotient group of $G_k$.

Finally, we can always add the free group of rank $2$ to the collection $\{G_k\}_{k \in \N}$. This will ensure that $\overline{G}$ can be generated by $2$ elements, as a
quotient of this free group.
\end{proof}

\begin{rem}\label{countrem}
Let us explain why the countability assumptions in Theorem \ref{main} cannot be dropped.
Observe that, given a simple group $S$, any non-trivial quotient of the free product
$S*S$ must contain a copy of $S$. So, if $S$ is an uncountable simple group (e.g., $PSL_2(\mathbb{R})$), then the acylindrically hyperbolic group $S*S$ cannot have a common
quotient with the free group of rank $2$, as $S$ is not a subgroup of any countable group.
On the other hand, in \cite{Camm} Camm showed that there exists a continuum of pairwise non-isomorphic $4$-generated infinite simple groups $S_i$, $i \in I$.
Note that the groups $G_i=S_i*S_i$ are acylindrically hyperbolic, $i \in I$, and any non-trivial common quotient of the family $\{G_i\}_{i \in I}$ will be
generated by $8$ elements and will
contain a copy of $S_j$, for each $j \in I$. But the latter is impossible since a countable group can have at most countably many finitely generated subgroups.
Thus the only common quotient of $\{G_i\}_{i \in I}$ is the trivial group.
\end{rem}

In view of Theorem \ref{main}, to obtain Corollaries \ref{cor1}--\ref{cor3} it suffices to prove the following.
We recall that a hyperbolic group is called {\it non-elementary} if it is not virtually cyclic.

\begin{lem}
Let $Q$ be any common quotient of all non-elementary hyperbolic groups. Then the following hold:
\begin{enumerate}
\item[(a)] $Q$ has property $FL^p$ for all $p\in [1, +\infty)$;
\item[(b)] every action of $Q$ on a finite dimensional contractible topological space has a fixed point;
\item[(c)] every simplicial action of $Q$ on a finite dimensional locally finite contractible simplicial complex is trivial;
\item[(d)] $Q$ is not uniformly non-amenable;
\item[(e)] for every sufficiently large $n\in \mathbb N$ and all $r\in \mathbb N$, there exists a finite generating set $X$ of $Q$ such that every element $g\in Q$ of length $|g|_X\le r$ has order at most $n$.
\end{enumerate}
\end{lem}

\begin{proof}
(a) For every $n\in \mathbb N$, there exists a non-elementary hyperbolic group $H$ such that for all $1\le p\le n$ every isometric action of $H$ on an $L^p$-space fixes a point. In fact, a random hyperbolic group, in a suitable density model, satisfies this property for every given $n$ almost surely \cite{DM,NS} (note that, by \cite[Theorem~1.3]{BFGM} and \cite[Corollary~D]{BGM}, property $FL^p$ for $p \in [1,2]$ follows from property $FL^2$). Since $FL^p$ passes to quotients, the result follows.

(b) Let $\mathcal{X}_n$ denote the class of contractible Hausdorff topological spaces of covering dimension $n$, $n \in \N$.
By \cite[Theorem 1.7]{Aetal} for each $n \in \N$ there exists a non-elementary hyperbolic group $G_n$ such that any action of $G_n$ by homeomorphisms on any space
$S \in \mathcal{X}_n$ has a global fixed point. Since $Q$ is a quotient of $G_n$, for every $n \in \N$, $Q$ does not admit a fixed point-free action on any space from
$\mathcal{X}=\bigcup_{n \in \N} \mathcal{X}_n$, as required.

(c) This can be derived using the same argument as in the proof of \cite[Corollary 1.2]{Aetal}. Indeed, first let us observe that the free product $Alt(n)*Alt(n)$, $n \ge 5$,
is a non-elementary hyperbolic group which does not contain any proper normal subgroups of index less than $n!/2$. Since $Q$ is a quotient of such groups for all $n \ge 5$, $Q$
cannot have any proper subgroups of finite index at all.
Now, if $Q$ acts on a finite dimensional simplicial complex $\mathcal{C}$, then, by claim (b), $Q$ should fix some point $p \in \mathcal{C}$. It follows that for every $R>0$ and $k \in \N\cup \{0\}$,  $Q$ acts on the set of $k$-simplices intersecting the ball of radius $R$ centered at $p$ in $\mathcal C$.
This set is finite, provided $\mathcal{C}$ is locally finite, thus, since $Q$ has no proper finite index subgroups, such an action must fix each simplex of $\mathcal C$.
Hence the action of $Q$ on $\mathcal{C}$ must be trivial.

(d) It was noted in \cite{Osi02b} that there exists a sequence of normal subgroups
$$
N_1\lhd N_2\lhd \ldots \lhd F,
$$
where $F$ is the free group of rank $2$ such that $H_i=F/N_i$ is non-elementary hyperbolic and $F/\bigcup_{i=1}^\infty N_i$ is amenable
(in fact, it is isomorphic to $\mathbb Z_2\, {\rm wr }\, \mathbb Z$.)
Let $X_i$ be the image of the standard basis of $F$ in $H_i$. It follows that the F\o{}lner constants of $H_i$ with respect to $X_i$ satisfy
$\mbox{F\o l} (H_i, X_i)\to 0$ as $i\to \infty$ \cite[Corollary 12.3]{A}. Let $Y_i$ be the image of $X_i$ under a surjective homomorphism $H_i\to Q$.
It is easy to see that F\o lner constants are monotone under epimorphisms, i.e., $\mbox{F\o l} (Q, Y_i)\le \mbox{F\o l} (H_i, X_i)$ in our notation
(see, for example, \cite[Theorem 4.1]{A}). Thus we have $\lim\limits_{i\to \infty} \mbox{F\o l} (Q, Y_i)=0$, i.e., $Q$ is not uniformly non-amenable.

(e) Let $n$ be a sufficiently large natural number. It follows from the work of Novikov-Adyan \cite{Ad} (for $n$ odd),  Ivanov \cite{I} and Lysenok \cite{Lys} (for $n$ even) and the isoperimetric characterization of hyperbolic groups that for every $r\in \mathbb N$, there exists a non-elementary hyperbolic group $H_r$, generated by a finite set $X_r$, such that every element $h\in H_r$ of length $|h|_{X_r}\le r$ has order at most $n$. This fact can also be found in an explicit form in the paper \cite{IO} (for all sufficiently large $n$). Given $r\in \mathbb N$, let $Y_r$ be the image of $X_r$ under the surjective homomorphism $H_r\to Q$. It is clear that every element $g\in Q$ of length $|g|_{Y_r}\le r$ has order at most $n$.
\end{proof}


\begin{thebibliography}{99}
\bibitem{Ad}
S.I. Adian, The Burnside problem and identities in groups. \emph{Ergebnisse der Mathematik und ihrer Grenzgebiete}, \textbf{95}. \emph{Springer-Verlag, Berlin-New York}, 1979.

\bibitem{AMS}
Y. Antolin, A. Minasyan, A. Sisto, Commensurating endomorphisms of acylindrically hyperbolic groups and applications, \emph{Groups Geom. Dyn.} \textbf{10} (2016), no. 4, 1149-1210.

\bibitem{Aetal}
G. Arzhantseva, M. Bridson, T. Januszkiewicz, I. Leary, A. Minasyan, J. Swiatkowski, Infinite groups with fixed point properties, \emph{Geom. Topol.} \textbf{13} (2009), no. 3, 1229-1263.

\bibitem{A}
G.N. Arzhantseva, J. Burillo, M. Lustig, L. Reeves, H. Short, E. Ventura, Uniform non-amenability, \emph{Adv. Math.} \textbf{197} (2005), no. 2, 499-522.

\bibitem{AAS}
J.W. Anderson, J. Aramayona, K.J. Shackleton, Uniformly exponential growth and mapping class groups of surfaces. In the tradition of Ahlfors-Bers. IV, 1–6, \emph{Contemp. Math.} \textbf{432}, Amer. Math. Soc., \emph{Providence, RI}, 2007.

\bibitem{DGO}
F. Dahmani, V. Guirardel, D. Osin, Hyperbolically embedded subgroups and rotating families in groups acting on hyperbolic spaces,
\emph{Memoirs AMS}  \textbf{245} (2017), no. 1156.

\bibitem{BFGM}
U. Bader, A. Furman, T. Gelander, N. Monod, Property (T) and rigidity for actions on Banach spaces,
\emph{Acta Math.} \textbf{198} (2007), no. 1, 57-105.

\bibitem{BGM} U. Bader, T. Gelander, N. Monod, A fixed point theorem for $L^1$ spaces, \emph{Invent. Math.} 189 (2012), no. 1, 143-148.

\bibitem{Camm} R. Camm, Simple free products. \emph{J. London Math. Soc.} \textbf{28} (1953), 66-76.

\bibitem{CD}
I. Chatterji, F. Dahmani, Proper actions on $\ell^p$ spaces for relatively hyperbolic groups.\\ \texttt{arXiv:1801.08047}

\bibitem{DM}
C. Dru{t}u, J.M. Mackay, Random groups, random graphs and eigenvalues of $p$-Laplacians, \emph{Adv. Math.} \textbf{341} (2019), 188-254.

\bibitem{G-dlH} \'E. Ghys, P. de la Harpe, Sur les groupes hyperboliques d'apr\`{e}s Mikhael Gromov, Progress in Mathematics, \textbf{83}. Birkh\"{a}user Boston, Inc., Boston, MA, 1990.

\bibitem{Gro}
M. Gromov, Hyperbolic groups, \emph{Essays in group theory}, 75-263,
Math. Sci. Res. Inst. Publ., \textbf{8}, \emph{Springer, New York}, 1987.

\bibitem{GST}
D. Gruber, A. Sisto, R. Tessera, Gromov's random monsters do not act non-elementarily on hyperbolic spaces. \texttt{arXiv:1705.10258}

\bibitem{Hae}
T. Haettel, Hyperbolic rigidity of higher rank lattices. \texttt{arXiv:1607.02004}

\bibitem{H}
U. Hamenst\"adt, Bounded cohomology and isometry groups of hyperbolic spaces,
\emph{J. Eur. Math. Soc.} \textbf{10} (2008), no. 2, 315-349.

\bibitem{Hull}
M. Hull, Small cancellation in acylindrically hyperbolic groups, \emph{Groups Geom. Dyn.} \textbf{10} (2016), no. 4, 1077-1119.

\bibitem{HO} M. Hull, D. Osin, Induced quasi-cocycles on groups with hyperbolically embedded subgroups, \emph{Alg. \& Geom. Topol.}, \textbf{13} (2013) 2635-2665.

\bibitem{I}
S.V. Ivanov, The free Burnside groups of sufficiently large exponents, \emph{Internat. J. Algebra Comput.} \textbf{4} (1994), no. 1-2.

\bibitem{IO}
S.V. Ivanov, A.Yu. Olʹshanskii,  Hyperbolic groups and their quotients of bounded exponents, \emph{Trans. Amer. Math. Soc.} \textbf{348} (1996), no. 6, 2091-2138.

\bibitem{KR}
I. Kapovich, K. Rafi, On hyperbolicity of free splitting and free factor complexes, \emph{Groups Geom. Dyn.} \textbf{8} (2014), no. 2, 391-414.

\bibitem{Kou}
M. Koubi, Croissance uniforme dans les groupes hyperboliques, \emph{Ann. Inst. Fourier}   \textbf{48} (1998), 1441-1453.

\bibitem{Lys}
I.G. Lysenok, Infinite Burnside groups of even period, \emph{Izv. Math.} \textbf{60} (1996), no. 3, 453-654.

\bibitem{Mann}
A. Mann, How groups grow. London Mathematical Society Lecture Note Series,\textbf{ 395}. Cambridge University Press, Cambridge, 2012.

\bibitem{MO}
A. Minasyan, D. Osin, Acylindrical hyperbolicity of groups acting on trees, \emph{Math. Ann.} \textbf{362} (2015), no. 3-4, 1055-1105.

\bibitem{MO-erratum} A. Minasyan, D. Osin, Erratum to the paper ``Acylindrical hyperbolicity of groups acting on trees''. \texttt{arXiv:1711.09486}

\bibitem{NS}
A. Naor and L. Silberman, Poincar\'e inequalities, embeddings, and wild groups, \emph{Compos. Math.} \textbf{147} (2011). no. 5, 1546-1572.

\bibitem{Osi02b}
D. Osin, Kazhdan constants of hyperbolic groups, \emph{Funct. Anal. Appl.} \textbf{36} (2002), no. 4, 290-297

\bibitem{Osi02}
D. Osin, Weakly amenable groups. \emph{Computational and statistical group theory (Las Vegas, NV/Hoboken, NJ, 2001)}, 105–113. {Contemp. Math.} \textbf{298},
\emph{Amer. Math. Soc., Providence, RI}, 2002.

\bibitem{Osi10}
D. Osin, Small cancellations over relatively hyperbolic groups and embedding theorems, \emph{Ann. Math.} {\bf 172} (2010), no. 1, 1-39.

\bibitem{Osi16} D. Osin, Acylindrically hyperbolic groups, \emph{Trans. Amer. Math. Soc.} \textbf{368} (2016), no. 2, 851-888.

\bibitem{Osi17} D. Osin, Groups acting acylindrically on hyperbolic spaces, \emph{Proc. Int. Cong. of Math.-- 2018,
Rio de Janeiro} \textbf{1} (2018), 915-936.

\bibitem{Ols93}
A.Yu. Olshanskii,  On residualing homomorphisms and
$G$-subgroups of hyperbolic groups, {\it Int. J. Alg. Comp.} {\bf 3}
(1993), no. 4, 365-409.

\bibitem{Sh}
Y. Shalom, Explicit Kazhdan constant for representations of semisimple and arithmetic groups,
\emph{Ann. Inst. Fourier} \textbf{50} (2000), no. 3, 833-863.

\bibitem{Xie}
X. Xie, Growth of relatively hyperbolic groups,
\emph{Proc. Amer. Math. Soc.} \textbf{135} (2007), no. 3, 695-704.

\bibitem{Y}
G. Yu, Hyperbolic groups admit proper affine isometric actions on $\ell^p$-spaces, \emph{Geom. Funct.
Anal.} \textbf{15} (2005), no. 5, 1144-1151.


\end{thebibliography}
\end{document}